\documentclass{article}

\usepackage[
	a4paper,
	margin=25mm
]{geometry}
\usepackage{setspace}
\usepackage{sectsty} 
\usepackage{titlesec}

\usepackage{amsmath, amssymb, amsthm}
\usepackage{mathtools}

\usepackage[
	setpagesize=false,
	colorlinks=true,
	linkcolor=blue,
        citecolor=blue,
        urlcolor=black,
	pdfencoding=auto,
	psdextra,
]{hyperref}
\usepackage{cleveref}

\usepackage{etoolbox} 
\usepackage[shortlabels]{enumitem}
\usepackage{xparse} 

\usepackage{graphicx}
\usepackage[dvipsnames]{xcolor}
\usepackage{tikz}

\usepackage{todonotes}
\usepackage[
    deletedmarkup=sout,
    commentmarkup=uwave,
    authormarkuptext=name
]{changes}

\usepackage{comment}

\setlist[enumerate,1]{label=(\arabic*), ref=(\arabic*)}
\setlist[enumerate,3]{label=(\roman*), ref=(\roman*)}

\allsectionsfont{\boldmath} 

\titlelabel{\thetitle.\quad}

\theoremstyle{plain}
\newtheorem{theorem}{Theorem}[section]
\newtheorem{lemma}[theorem]{Lemma}

\newtheorem{problem}[theorem]{Problem}

\newtheorem{claim}{Claim}[theorem]
\newtheorem*{claim*}{Claim}

\makeatletter
\newenvironment{claimproof}[1][Proof]{\par
	\pushQED{\qed}%
	
	\normalfont \topsep6\p@\@plus6\p@\relax
	\trivlist
	\item[\hskip\labelsep
	\textit{#1}\@addpunct{.}~]\ignorespaces
}{%
	\popQED\endtrivlist\@endpefalse
}
\makeatother

\newlist{Cases}{enumerate}{3}
\setlist[Cases]{parsep=0pt plus 1pt}
\setlist[Cases,1]{wide=0pt, listparindent=\parindent,
    label = \textbf{Case~\arabic*:}, ref = \arabic*}
\setlist[Cases,2]{wide=\parindent, listparindent=\parindent,
    label = \textbf{Case~\arabic{Casesi}-\arabic{Casesii}:}}

\crefname{Casesi}{case}{cases}

\newcounter{case}
\AtBeginEnvironment{proof}{\setcounter{case}{0}}

\crefname{case}{case}{cases}

\theoremstyle{definition}
\newtheorem{definition}[theorem]{Definition}

\newenvironment{proofsketch}[1][Proof sketch]{%
	\begin{proof}[#1]
}{%
	\end{proof}
}

\newcommand{\labs}{\Big\lvert}
\newcommand{\rabs}{\Big\rvert}

\makeatletter
\NewDocumentCommand{\xsideset}{mmme{_^}}{%
  \mathop{%
    \settowidth{\dimen0}{$\m@th\displaystyle#3$}%
    \dimen0=.5\dimen0
    \settowidth{\dimen2}{$%
      \m@th\displaystyle#3%
      \IfValueT{#4}{_{#4}}%
      \IfValueT{#5}{^{#5}}%
    $}%
    \dimen2=.5\dimen2
    \advance\dimen2 -\dimen0
    \sbox6{\scriptspace\z@$\displaystyle{\vphantom{#3}}#1$}
    \sbox8{\scriptspace\z@$\displaystyle{\vphantom{#3}}#2$}
    \ifdim\wd6>\dimen2 \kern\dimexpr\wd6-\dimen2\relax\fi
    {%
     \mathop{\llap{\copy6}{\displaystyle#3}\rlap{\copy8}}\limits
     \IfValueT{#4}{_{#4}}%
     \IfValueT{#5}{^{#5}}%
    }%
    \ifdim\wd8>\dimen2 \kern\dimexpr\wd8-\dimen2\relax\fi
  }%
}
\makeatother

\let\originalleft\left
\let\originalright\right
\renewcommand{\left}{\mathopen{}\mathclose\bgroup\originalleft}
\renewcommand{\right}{\aftergroup\egroup\originalright}

\makeatletter
\newcases{lrcases*}
  {\quad}
  {$\m@th{##}$\hfil}
  {{##}\hfil}
  {\lbrace}
  {\rbrace}
\makeatother

\usetikzlibrary{matrix,arrows,decorations.pathmorphing}
\usetikzlibrary{positioning}
\usetikzlibrary{graphs} 
\usetikzlibrary{graphs.standard}

\title{Dense triangle-free $(n, d, \lambda)$-graphs for all orders}

\author{
Jaehoon Kim~\thanks{Department of Mathematical Sciences, KAIST, South Korea Email:{\tt $\{$jaehoon.kim, hyunwoo.lee$\}$@kaist.ac.kr}} \and Hyunwoo Lee~\footnotemark[1]~\thanks{Extremal Combinatorics and Probability Group (ECOPRO), Institute for Basic Science (IBS).}}

\usepackage[
    backend=biber,
    sorting=nyt,
    maxbibnames=99
]{biblatex}
\begin{document}
\maketitle

\begin{abstract}
In \cite{alon1994explicit}, Alon construct a triangle-free $(n,d,\lambda)$-graph with $d = \Omega(n^{2/3})$ and $\lambda = O(d^{1/2})$ for an exponentially increasing sequence of integers $n$.
Using his ingenious construction, we deduce that there exist triangle-free $(n,d,\lambda)$-graphs with $d = \Omega(n^{2/3})$ and $\lambda = O( (d \log n)^{1/2} )$ for all sufficiently large $n$.
\end{abstract}


\section{Introduction}\label{sec:intro}

Since Erd\H{o}s and R\'{e}nyi introduced the concept of random graphs, 
Erd\H{o}s-R\'{e}nyi random graph models and their variations have been extensively studied and proven to be extremely useful. 
For example, various random graph models provide effective tools for analyzing various real-life phenomena (for such examples, see \cite[Chapter 11]
{coolen2017generating}) and they often provide good examples and bounds for extremal problems.

As Erd\H{o}s-R\'{e}nyi random graph provides highly unstructured graphs, it is natural to consider other models that produce random graphs admitting certain structures.
By imposing restrictions on the degrees of vertices, the study of random regular graphs was initiated by Bender and Canfield~\cite{bender1978asymptotic}, Bollob\'as~\cite{bollobas1980probabilistic}, and Wormald~\cite{wormald1981asymptotic,wormald1981asymptotic-b}. 
Introducing restrictions on its local structures, random regular graphs offer a model that lies on the cusp between structured and unstructured graphs. 
A model with such a local restriction is useful for various problems. For example, there is an application of random regular graphs even in biogeography in \cite{wilson1987methods}.

As these random graph models provide interesting graphs only in a non-constructive way, the need for considering a specific graph that behaves like a random graph has emerged, giving birth to the concept of pseudorandom graphs. 
The concept of pseudorandom graphs was proposed by Thomason~\cite{Th87a,Th87b} in the 1980s, and it has become a central topic in both combinatorics and theoretical computer science. Among the several notions for pseudorandomness, the most frequently used ones are the following jumbledness and bijumbledness introduced by Thomason~\cite{Th87a,Th87b}. 

\begin{definition}
    For given positive real numbers $p, \beta$, a graph $G$ is $(p, \beta)$-jumbled if $\lvert e(X) - p\binom{|X|}{2} \rvert \leq \beta |X|$ holds for every $X\subseteq V(G)$.
    It is $(p, \beta)$-bijumbled if  $\lvert e(S, T) - p|S||T| \rvert \leq \beta \sqrt{|S||T|}$ holds for every $S, T\subseteq V(G)$.
\end{definition}

Note that bijumbledness implies jumbledness, while the converse is not true. 
For more about pseudorandom graphs, we recommend an excellent survey paper by Krivelevich and Sudakov~\cite{krivelevich2006pseudo}.

An important class of pseudorandom graphs are $(n, d, \lambda)$-graphs. 
For a graph $G$ with eigenvalues $\lambda_1\geq \dots \geq \lambda_n$ of its adjacency matrix, write 
$\lambda(G)= \max\{ |\lambda_2|,|\lambda_n|\}$. 
An $(n, d, \lambda)$-graph is a $d$-regular $n$-vertex graph whose second largest absolute value of eigenvalues is $\lambda$.  
By the following famous Expander mixing lemma, an $(n,d,\lambda)$-graph is $(\frac{d}{n},\lambda)$-bijumbled, hence $(n,d,\lambda)$-graphs provide an interesting class of pseudorandom graphs.
\begin{lemma}[Expander Mixing Lemma]\label{lem:expander-mixing}
    Let $G$ be an $(n, d, \lambda)$-graph. Then 
    $\Big\lvert e(S, T) - \frac{d}{n}|S||T| \Big\rvert \leq \lambda \sqrt{|S||T|}$ holds for all $S, T\subseteq V(G)$.
\end{lemma}

Moreover, through several concepts like Cayley graphs and random walks, the notion of $(n, d, \lambda)$-graph constitutes solid brides between combinatorics, algebra, geometry, and probability. For this interesting interplay, see an excellent Survey by Hoory, Linial, and Wigderson~\cite{hoory2006expander}.

While Alon-Boppana bound~\cite{nilli1991second} showed that an $(n,d,\lambda)$-graph always satisfies $\lambda =\Omega(d^{1/2})$, many interesting applications (see~\cite{hoory2006expander}) of $(n,d,\lambda)$-graph become most powerful when $\lambda$ is $O(d^{1/2})$. Thus finding $(n,d,\lambda)$-graphs with $\lambda =O(d^{1/2})$ possessing additional interesting properties is a crucial topic of study.
For example, as the random regular graphs model imposes the degree-regularity to make the graph more structured, can we impose further conditions on $(n,d,\lambda)$-graphs to ensure more structures? 

By enforcing the additional triangle-free condition, Alon \cite{alon1994explicit} constructed triangle-free $(n,d,\lambda)$-graphs with $d=\Theta(n^{2/3})$ and $\lambda=O(d^{1/2})$. 
As the expander mixing lemma implies that every $(n, d, \lambda)$-graph with $\lambda < d^2/n$ contains a triangle, his construction is as dense as it gets. 
Moreover, Alon's construction has several surprising applications in geometry, extremal graph theory as well as coding theory (see a survey by Alon~\cite{alon2020lovasz}).
Alon's construction is indeed one which lies on the cusp of structured and unstructured worlds, as a Erd\H{o}s-R\'{e}nyi random graphs $G(n,p)$ with degree $pn=\Theta(n^{2/3})$ has triangles everywhere. We formally state his theorem as below.
 
\begin{theorem}[Alon~\cite{alon1994explicit}]\label{thm:alon}
    Let $k$ be a positive integer which is not divisible by $3.$ Then there is a triangle-free $(n, d, \lambda)$-graph where $n = 2^{3k}$, $d = 2^{k-1}(2^{k-1} - 1)$, and $\lambda \leq 9\cdot 2^{k} + 3\cdot 2^{k/2} + 1/4.$
\end{theorem}

Among two local restrictions, one can give up degree-regularity while retaining triangle-freeness.
Indeed, Conlon~\cite{conlon2017sequence} found another random construction of triangle-free pseudorandom graphs as follows.

\begin{theorem}[Conlon~\cite{conlon2017sequence}]\label{thm:conlon}
    There exists a sequence of positive integers $\{n_i\}_{i = 1}^{\infty}$ such that for each $i$, there is a triangle-free $n_i$-vertex graph $G_i$ which is $(p, \beta)$-jumbled with $p = \Omega(n_i^{-1/3})$ and $\beta = O(\sqrt{pn_i}\log n_i).$
\end{theorem}
Unlike Alon's construction, Conlon's construction is not explicit and not regular. However, it is more intuitive as to why it is triangle-free and pseudorandom. Also, note that the value of $\beta$ above is $\log n_i$ factor away from the optimal bound, but this log factor might just be an artifact of the proof method.

Notice that Alon's constructions provide $n$-vertex graphs for an exponentially increasing sequence of $n$, leaving out many values of $n$. 
Although one can simply take an arbitrary $n$-vertex induced subgraph of the construction to obtain a pseudorandom triangle-free graph and the eigenvalue interlacing theorem guarantees its optimal pseudorandomness, such a graph would not be regular. 
Indeed, it is very interesting to ask whether one can ensure the criterion of pseudorandomness as well as two local restrictions of degree-regularity and triangle-freeness.

\begin{problem}
    Is there a triangle-free $(n, d, \lambda)$-graph where $d = \Omega (n^{2/3})$ and $\lambda = O(d^{1/2})$ for all $n$?
\end{problem}

We prove that this is true if we allow an additional $\log^{1/2} n$ factor at the eigenvalue bound $\lambda$.

\begin{theorem}\label{thm:main}
For all sufficiently large integer $n$, there is a triangle-free $(n, d, \lambda)$-graph where $d\geq \frac{1}{20} n^{2/3}$ and $\lambda \leq 200(d \log n)^{1/2}.$ 
\end{theorem}

In order to prove this theorem, we take a random induced subgraph $G[X]$ of Alon's construction and further delete its edges to make it regular while making sure that $\lambda(G[X])$ does not increase too much while the edge deletions. Although other parts of our proof can be made constructive\footnote{We use randomness for the function $g$ in the proof of \Cref{lem:sponge-graph}, but it is easy to take desired collections $\mathcal{C}_v$ in a constructive way.}, taking an induced subgraph $G[X]$ with $\Delta(G[X])-\delta(G[X]) \leq O(n^{1/3}\log^{1/2}n)$ is not constructively done.


\section{Preliminaries}\label{sec:prelim}

\subsection{Notations}\label{subsec:notations}

We denote by $[n]$ the set $\{1, \dots, n\}.$ If we claim that a statement hold when $\gamma_1, \gamma_2 \ll \alpha, \beta$, it means there is a function $f$ and $g$ such that the statement holds for all $\gamma_1 \leq f(\alpha, \beta)$ and $\gamma_2 \leq g(\alpha, \beta)$. We do not explicitly compute these functions $f$ and $g.$ For our convenience, we omit the floor and ceilings for large numbers and treat them as integers. 

In this paper, we use general notations in graph theory. For a graph $G$, we denote by $V(G)$ and $E(G)$ the set of vertices and edges of $G$, respectively. For a vertex $v\in V(G)$, we write $N_G(v)$ as the set of neighbors of $v$ in $G.$ We denote $|N_G(v)|$ as $d_G(v)$, the degree of $v$ in $G.$ For two vertex sets $S,T$, $e_G(S,T)$ counts the pair $s\in S$ and $t\in T$ with $st\in E(G)$. In particular, $e_G(V(G),V(G)) = 2|E(G)|$.
If $G$ is clear in the context, we may omit the subscript from these notations. We write $\delta(G)$ and $\Delta(G)$ as the minimum and maximum degree of $G$, respectively. 
Let $X, S, T\subseteq V(G).$ We denote by $G[X]$ the induced subgraph of $G$ induced by the vertex subset $X.$


\subsection{Useful lemmas}\label{subsec:useful-lemmas}
The following Chernoff bound will be useful for us.

\begin{lemma}[Chernoff bound~\cite{janson2011random}]\label{lem:chernoff}
    Let $X$ be either a sum of $n$ independent Bernoulli random variables or a random variable following a hypergeometric distribution with parameters $N,K$, and $n$. Then for every $\delta \in [0,1)$, we have 
    $Pr\left(|X - \mathbb{E}[X]| \geq \delta \mathbb{E}[X] \right) \leq 2 exp\left(-\frac{\delta^2 \mathbb{E}[X]}{3}\right).$
\end{lemma}

The following two lemmas provide some controls on pseudorandomness after edge deletions.

\begin{lemma}\label{lem:new-bijumbledness}
    Let $G$ be an $n$-vertex $(p, \beta)$-bijumbled graph and $F\subseteq G[X]$ with $X\subseteq V(G)$.
    Then the graph $G[X]\setminus F$ is  $(p, \beta + \Delta(F))$-bijumbled.
\end{lemma}

\begin{proof}
    For $S,T\subseteq X$ with $|S|\leq |T|$, we have 
    \begin{align*}
        \labs e_{G[X]\setminus F}(S, T) - p|S||T| \rabs 
        &\leq  \labs e_{G}(S, T) - p|S||T| \rabs + \labs e_{F}(S, T) \rabs \\
        &\leq  \beta \sqrt{|S||T|} + \Delta(F)|S|
        \leq (\beta + \Delta(F))\sqrt{|S||T|} .
   \end{align*}
    Thus $G[X]\setminus F$ is $(p, \beta + \Delta(F)$-bijumbled.  
\end{proof}

\begin{lemma}\label{lem: obtain lambda}
    Let $G$ be an $(N,d,\lambda)$-graph and $X\subseteq V(G)$ be a vertex subset of size $n$. Suppose that $F$ is a subgraph of $G[X]$ and $G[X] \setminus F$ is a $d'$-regular graph.
    Then $G[X]\setminus F$ is an $(n,d',\lambda + \Delta(F))$-graph.
\end{lemma}

\begin{proof}
    Let $A$ be the adjacent matrix of $G$ where the first $n$ rows and columns correspond to the vertices in $X$, and $A_X, A_F$ be the adjacent matrix of $G[X]$ and $F$, respectively.
    Let $\mathbf{1}$ be the column vector in $\mathbb{R}^N$ with all entries $1$, and let $\mathbf{1}_X$ be the column vector in $\mathbf{R}^n$ with all entries $1$.

    By using Courant-Fischer Theorem, there exists a unit vector $\mathbf{v}$ perpendicular to $\mathbf{1}_X$ with 
    $|\mathbf{v}^T (A_X - A_F) \mathbf{v}| = \lambda (G[X]\setminus F)$. 
    Let $\mathbf{u}$ be the vector obtained by appending $N-n$ zeros at the end of $\mathbf{v}$. As $\mathbf{u} \perp \mathbf{1}$, Couran-Fischer Theorem implies $ |\mathbf{v}^T A_X \mathbf{v}^T|=  |\mathbf{u}^T A \mathbf{u}| \leq \lambda.$
    On the other hand, we have
    \begin{align*}
        \lambda(G[X]\setminus F)  = |\mathbf{v}^T (A_X -A_F) \mathbf{v}^T|\leq |\mathbf{v}^T A_X \mathbf{v}|  + |\mathbf{v}^T A_F \mathbf{v}| \leq \lambda + \Delta(F).
    \end{align*}
    This shows that $G[X]\setminus F$ is an $(n,d',\lambda+\Delta(F))$-graph.
\end{proof}

As we plan to delete some edges from a graph to make it regular, we need to take a subgraph of a specific degree sequence in a pseudorandom graph. The following lemma allows us to find a subgraph with specific degrees on a given set while having small degrees on all other vertices.

\begin{lemma}\label{lem: Hall}
Let $0< 1/n\ll 1.$
    Let $G$ be an $n$-vertex $(p,\beta)$-bijumbled graph with $p\geq n^{-1/3}\log^{-1} n, \beta =n^{1/3}\log n$, and $\delta(G)\geq \frac{1}{2} p n$ and let $X\subseteq V(G)$ be an independent set in $G$.
    For $d=n^{1/3} \log n$ and a given function $f:X\rightarrow [0,d]$, there exists a subgraph $F$ of $G$ satisfying the following.

    \begin{itemize}
        \item For each $x\in X$, $d_F(x)=f(x)$.
        \item For each $y\in V(G)\setminus X$, $d_F(y)\leq \log^{10} n$ 
    \end{itemize}
\end{lemma}

\begin{proof}
    As $X$ is an independent set and $G$ is $(p,\beta)$-bijumbled, we have $0=e_G(X)\geq \frac{1}{2}p|X|^2- \beta |X|$, meaning that $|X| \leq 2n^{2/3} \log^2 n $.

    Let $Y=V(G)\setminus X$ and consider the bipartite graph $G[X,Y]$ between $X$ and $Y$.
    We duplicate each vertex $x\in X$ into $d$ copies with the same neighborhood as $x$ and duplicate each vertex $y$ in $Y$ into $\log^2 n$ copies with the same neighborhood as $y$. Let ${\rm Aux}$ be the new auxiliary graph obtained with bipartition $(X',Y')$ where $X',Y'$ are the sets of duplicates. We claim that ${\rm Aux}$ satisfies the Hall's condition, hence containing a matching covering $X'$. For a nonempty subset $A'\subseteq X'$, we have a corresponding subset $A\subseteq X$ of original vertices with size at least $|A|\geq |A'|/d$.

    If $|A|\leq \frac{\beta^2 \log n}{ p^2 n} =  n^{1/3} \log^5 n $, then the minimum degree condition on $G$ ensures 
    $$|N_{{\rm Aux}}(A')|\geq \delta(G)\cdot \log^{10}n \geq \frac{1}{2} pn \log^{10}n \geq n^{1/3} \log n \cdot  n^{1/3} \log^5 n = d |A| \geq |A'|.$$
    If $|A|> \frac{\beta^2 \log n}{ p^2 n} = n^{1/3} \log^5 n$, then we have 
    $|N_{G[X,Y]}(A)|\geq n/4$, as otherwise the set $C\subseteq Y\setminus N_{G[X,Y]}(A)$ of size $n/4$ satisfies 
    $$e_{G[X,Y]}(A,C) \geq  p|A||C|-\beta\sqrt{|A||C|} = \sqrt{|A||C|}\left(   n^{1/3} \log^{3/2} n - n^{1/3} \log n \right)>0,$$
    a contradiction. Hence, as $|X|\leq 2 n^{2/3} \log^2 n$, we have $$|N_{{\rm Aux}}(A')|\geq |N_{G[X,Y]}(A)|\cdot \log^{10}n \geq \frac{1}{4} n \log^{10} n \geq d|X| =|X'|\geq |A'|.$$
    Thus Hall's condition holds for all $A'\subseteq X'$ and Hall's theorem implies that ${\rm Aux}$ contains a matching covering $X'$. It is straightforward from the definition of ${\rm Aux}$ that this corresponds to a graph $F_0$ on a bipartite graph $G[X,Y]$ where all vertices in $X$ have degree $d$ and all vertices in $Y$ has degree at most $\log^{10}n$. For each $x$, we arbitrarily delete exactly $d-f(x)$ edges incident with it to obtain graph $F$, then it is obvious that this is a desired graph.
\end{proof}

We also need the following lemma. We believe that the proofs of several known results (for example, results in ~\cite{krivelevich2006pseudo,hyde2023spanning}) without any modifications imply this lemma. However, those results are all stated for the $(n,d,\lambda)$-graphs and we don't want to assume the degree-regularity in the following statement. Hence we provide a sketch of its easy proof.

\begin{lemma}\label{lem: spanning tree}
Let $0<1/n\ll 1$.
    Let $G$ be an $n$-vertex $(p,\beta)$-jumbled graph with $p\geq n^{-1/3}\log^{-1}n$, $\beta = n^{1/3} \log n$ and $\delta(G)\geq \frac{1}{2}p n$. Then $G$ contains a spanning tree $T$ with $\Delta(T)\leq 10$.
\end{lemma}
\begin{proofsketch}
Take a largest tree $T_0$ of $G$ with maximum degree at most $9$.
 It is easy to see that $|T_0| = \Omega(n)$ (for example, Friedman-Pippenger~\cite{friedman1987expanding} with the bijumbledness implies this) and $T$ contains has at least $8|T_0|/9$ vertices of degree less than $9$.
    Then we have $B=V(G)\setminus V(T_0)$ has size at most $n^{1/2}$ and $|V(T_0)|\geq n- n^{1/2}$ as otherwise bijumbledness of $G$ provides an edge between $B$ and the set of vertices in $T_0$ with degree less than $9$, and adding it to $T_0$ contradicts the maximality of $T_0$.

    As $n^{1/2} = \omega(pn)$, the minimum degree condition and bijumbledness implies that
    the bipartite graph $G[B,V(G)\setminus B]$ satisfies Hall's condition.
    So there is a matching in $G[B,V(T_0)]$ saturating $B$ and
    the tree $T_0$ together with this matching provides a desired spanning tree.
 \end{proofsketch}


\section{Degree adjusting graphs}\label{sec:adjusting-graph}
In this section, we prove that a pseudorandom graph contains a pair of `degree adjusting graphs'.
Roughly speaking, the following lemma ensures that we can find two graphs $R$ and $S$ in a pseudorandom graph such that $R\cup S$ contains a subgraph with any degree sequence which is dominated by the degree sequence $\mathbf{d}$ of $S$ as long as it is not too far from $\mathbf{d}$. This lemma will be useful for ensuring that all the vertices of our graph will have exactly the same degree.

\begin{lemma}\label{lem:sponge-graph}
    Suppose $0 < \frac{1}{n} \ll 1.$
    Let $G$ be an $n$-vertex triangle-free $(p, n^{1/3}\log n)$-bijumbled graph with $ n^{-1/3} \log^{-1}n \leq p \leq n^{-1/3} \log n$ and  $\delta(G)\geq \frac{1}{2}pn.$ 
    Then there are two edge-disjoint subgraphs $R$ and $S$ such that the following holds.
    \begin{enumerate}[label=\rm (RS\arabic*)]
        \item \label{RS1} $\Delta(R\cup S) \leq 600 n^{1/10}$,
        \item \label{RS2} $d_{S}(v) \geq 2 n^{1/10}$ for each $v\in V(G)$.
        \item \label{RS3} For every function $f:V(G)\rightarrow [n^{1/10}]$, there is a subgraph $H\subseteq R\cup S$ with $d_H(v)= d_{S}(v) -2f(v)$ for each $v\in V(G)$.
    \end{enumerate}
\end{lemma}

\begin{proof}
    Note that $C_5$ can be decomposed into $2K_2$ and $P_3\sqcup K_2$ and the degree sequence of those two graphs differ only at one vertex by two. If we put $P_3\sqcup K_2$ into $S$ and $2K_2$ into $R$, then 
    deleting $P_3\sqcup K_2$ from $S$ and adding $2K_2$ will lower the degree of the vertrex by two while keeping the rest of the degree same.
    We utilize this simple fact to prove the lemma, by supplying enough number of such $C_5$s at every vertex.
 The following claim will be useful for collecting many $C_5$s at each vertex.

    \begin{claim}\label{clm:v-pentagons}
        Let $G'=G\setminus F$ for a subgraph $F$ of $G$ with $\Delta(F) \leq n^{1/3}$.
        Let $X \subseteq V(G)$ be a vertex set of size at least $\frac{n}{2}$ and $v\in V(G)$ be a vertex having at least $\frac{1}{5} pn$ neighbors in $X$ in the graph $G'$.
        Then $G'[X\cup \{v\}]$ contains at least $\frac{1}{100}pn$ copies of $C_5$s where every two copies of $C_5$ shares exactly $\{v\}$.
    \end{claim}

    \begin{claimproof}
        Note that \Cref{lem:new-bijumbledness} together with the fact $\Delta(F) \leq n^{1/3}$ implies that $G'$ is a $(p, 2n^{1/3}\log n)$-bijumbled triangle-free graph.
        Consider a maximum collection $\mathcal{C}$ of $C_5$s as desired, and let $X'\subseteq X$ be the set of vertices in $X$ which does not belong to any copy of $C_5$ in $\mathcal{C}$.

        Suppose $|\mathcal{C}| < \frac{1}{100} pn $. 
        Then $X'$ contains at least $n/2 - \frac{4}{100} pn  \geq n/4$ vertices and $v$ has at least $pn  - \frac{4}{100} pn  \geq \frac{1}{6}pn$ neighbors in $X'$.
        For a subset $U = N_{G'}(v)\cap X'$ of size exactly $\frac{1}{6}pn$, consider $X'\setminus N_{G'}(U)$. Then we have $|X'\setminus N_{G'}(U)|< n/8$ as otherwise we can take a subset $U'\subseteq X'\setminus N_{G'}(U)$ of size exactly $n/8$ satisfying
        \begin{align*} 
            0 &= e_{G'}(U,U') \geq p|U||U'| - 2 n^{1/3}\log n \cdot \sqrt{|U||U'|} \\
            &= \frac{1}{48}p^2 n^2   - 
            \frac{1}{2\sqrt{3}} p^{1/2} n^{4/3} \log n \geq \frac{1}{48} n^{4/3} \log^{-2} n - \frac{1}{2\sqrt{3}} n^{7/6} \log^{2} n >0,
        \end{align*}
        a contradiction.
    
        Hence $N_{G'}(U)$ has size at least $n/4-n/8\geq n/8$, bijumbledness implies that $G'$ contains an edge in $N_{G'}(U)$. As $G'=G\setminus F$ is triangle-free, this edge corresponds to a copy of $C_5$ containing $v$ and four vertices in $X'\subseteq X$. As the vertices $X'$ is not in any copy of $C_5$ in $\mathcal{C}$, this contradicts the maximality of $\mathcal{C}$, hence proves $|\mathcal{C}| \geq \frac{1}{100} pn$. This proves the claim. 
    \end{claimproof}

    We now wish to consider a collection of $C_5$s which supplies enough copies of $C_5$ for each vertex, while their union forms a graph with not too large maximum degree.
    Let $\mathcal{C}_0$ be the collection that minimizes $y$ among the collections of edge-disjoint copies of $C_5$ in $G$ that satisfies the following where $F_0:= \bigcup_{C\in \mathcal{C}_0}C$ is the union of all graphs in $C$ and $y$ is the number of vertices $v\in V(G)$ with $d_{F_0}(v) < 20 n^{1/10}$.

    \begin{enumerate}[label=\rm (C\arabic*)]
        \item\label{C1} $|E(F_0)|\leq 50 n^{1/10}(n-y)$.
        \item\label{C2} $\Delta(F_0)\leq 300 n^{1/10}$.
    \end{enumerate}
    Indeed, $\mathcal{C}_0=\emptyset$ satisfies \ref{C1} and \ref{C2}, so such a choice exists.  
    Fix such a choice $\mathcal{C}_0$ minimizing $y$.
    By \ref{C2} and \Cref{lem:new-bijumbledness}, $G'=G\setminus F_0$ is $(p,2n^{1/3}\log n)$-bijumbled.
    Let $X$ be the set of vertices $v$ with $d_{F_0}(v)\leq  300 n^{1/10} - 2$.
    Then we have $|X|\geq \frac{4}{5}n$ as  
    $$ n-|X|\leq \frac{|E(F_0)|}{300 n^{1/10}-1}\leq \frac{1}{5}n.$$
    We now claim that $y\leq n^{1/3} \log^5 n$ holds. 
    Otherwise, we let $Y$ be a set of exactly $n^{1/3} \log^5 n$ vertices $v$ with $d_{F_0}(v)< 20 n^{1/10}$.
    Then we have 
    \begin{align*}
        e_{G'}(X,Y) \geq p|X||Y| - 2n^{1/3}\log n \cdot \sqrt{|X||Y|} \geq (p|X| - 2 n^{2/3} \log^{-3/2} n )|Y| \geq \frac{1}{2} p|X| |Y|,
    \end{align*}
    as $p> n^{1/3} \log^{-1}n$ and $|X|\geq \frac{4}{5} n$.
    This shows that $Y$ contains a vertex $u$ having at least $\frac{1}{2}pn$ neighbors in $X$.
    By \Cref{clm:v-pentagons}, $G'[X\cup \{v\}]$ contains $10 n^{1/10} \leq \frac{1}{100} pn$ copies of $C_5$ where they pairwisely share only $v$.
    By adding them to $\mathcal{C}_0$, we obtain a new collection of $C_5$ with a smaller value of $y$.
    Moreover, as this addition increases the number of edges in $F$ by $5\times 10n^{1/10}$, this new collection satisfies \ref{C1}, and the choice of $X$ ensures that it also satisfies \ref{C2}. As this is a contradiction to the minimality of $y$ and $\mathcal{C}_0$, we have $y=|Y|<n^{1/3} \log^{5}n$.

    Now, we consider another collection $\mathcal{C}_1$ that minimizes $y'$ among the collections of edge-disjoint copies of $C_5$ in $G'=G\setminus F_0$ that satisfies the following where $F_1:= \bigcup_{C\in \mathcal{C}_1}C$ is the union of all graphs in $\mathcal{C}_1$ and $y'$ is the number of vertices $v\in Y$ with $d_{F_1}(v) < 20 n^{1/10}$.

    \begin{enumerate}[label=\rm (C$'$\arabic*)]
        \item\label{C'1} $|E(F_1)|\leq 50 n^{1/10}(|Y|-y')$.
        \item\label{C'2} $\Delta(F_1)\leq 300 n^{1/10}$.
    \end{enumerate}
    Indeed, $\mathcal{C}_1=\emptyset$ satisfies \ref{C'1} and \ref{C'2}, so such a choice exists. With such a choice, we claim $y'=0$. If not, we have a vertex $v\in Y$ with $d_{F_1}(v)< 20n^{1/10}.$
    Let $Z$ be the set of vertices $u$ with $d_{F_1}(u)\leq 300 n^{1/10}-2$. Then we have 
    $$n - |Z| \leq \frac{|E(F_1)|}{ 300 n^{1/10}-1}\leq \frac{25 n^{1/3+ 1/10} \log^5 n }{300 n^{1/10}-1} < \frac{1}{10} pn.$$
    Thus, $v$ has at least $\delta(G)- \frac{1}{10}pn \geq \frac{1}{2} pn$ neghbors in $Z$ and \Cref{clm:v-pentagons} implies that $G[Z\cup \{v\}]$ contains at least $10 n^{1/10}$ edge-disjoint copies of $C_5$s pairwisely sharing only $v$. By adding these copies to $\mathcal{C}_1$, we obtain a new collection with a smaller value of $y'$.
    As this adds at most $50 n^{1/10}$ edges, the new collection satisfies \ref{C'1}, and the definition of $Z$ ensures that it also satisfies \ref{C'2}.
    This proves $y'=0$.

    Now, we take $\mathcal{C}=\mathcal{C}_0\cup \mathcal{C}_1$ and $F=F_0\cup F_1$ then \ref{C2}, \ref{C'2} and the fact that $y'=0$ implies that it satisfies the following.
    \begin{enumerate}[label=\rm (P\arabic*)]
        \item\label{P1} For each $v\in v(G)$,  $d_{F}(v) \geq 20 n^{1/10}$.
        \item\label{P2} $\Delta(F)\leq 600 n^{1/10}$.
    \end{enumerate}

    Now, for each copy $C\in \mathcal{C}$ of $C_5$, we choose a vertex $g(C)$ in $V(C)$ uniformly at random, and partition $\mathcal{C}$ into the sets $\{\mathcal{C}_v: v\in V(G)\}$ with
    $$\mathcal{C}_v = \{ C : g(C)=v \}.$$
    As $\mathbb{E}[|\mathcal{C}_v|] \geq \frac{1}{5}\cdot \frac{1}{2}d_F(v) \geq 2 n^{1/10}$, \Cref{lem:chernoff} with the union bound implies that 
    $|\mathcal{C}_v|> n^{1/10}$ holds for all $v\in V(G)$ with probability at least $1- n e^{n^{1/10}/100} >0$.
    Hence we can fix one such partition satisfying $|\mathcal{C}_v|> n^{1/10}$ for all $v\in V(G)$.

    For each $C\in \mathcal{C}$, we let $R_C$ be the copy of $2K_2$ avoiding the vertex $g(C)$ and $S_C$ be the subgraph consisting of the remaining three edges in $C$.
    Let $$R=\bigcup_{C\in \mathcal{C}} R_C \text{ and } S=\bigcup_{C\in \mathcal{C}} S_C.$$
    With this choice, as $F=R\cup S$, \ref{P2} implies that \ref{RS1} holds.
    Also, for each $v\in V(G)$, we have $d_{S}(v) \geq 2|\mathcal{C}_v| \geq 2n^{1/10}$, thus \ref{RS2} holds as well.
    Finally, for a given function $f:V(G)\rightarrow [n^{1/10}]$, we arbitrarily partition each $\mathcal{C}_v$ into 
    $\mathcal{C}_v^0$ and $\mathcal{C}_v^1$ with $|\mathcal{C}_v^1|= f(v)$. 
    This is possible as $|\mathcal{C}_v|\geq n^{1/10}$.
    Then the graph 
    $$H = \bigcup_{v\in V(G)}\bigcup_{C\in \mathcal{C}_v^0} S_C \cup \bigcup_{v\in V(G)}\bigcup_{C\in \mathcal{C}_v^1} R_C $$
    satisfies \ref{RS3}.
    Indeed, for all $u\neq v \in V(G)$ with $C\in \mathcal{C}_u$, switching $S_C$ wiht $R_C$ does not affect the degree of $v$. On the other hand, for each $C\in \mathcal{C}_v$, switching $S_C$ with $R_C$ lowers the degree of $v$ by two. Hence we have 
    $d_H(v) = d_S(v) - 2 |\mathcal{C}_v^1| = d_S(v)-2f(v).$
    This verifies \ref{RS3} and proves the lemma.
\end{proof}


\section{Proof of \Cref{thm:main}}\label{sec:proof}

In this section, we prove our main theorem.

\begin{proof}[Proof of \Cref{thm:main}]
For a given sufficiently large $n$, let $k$ be the smallest integer not divisible by $3$ with $n \leq 2^{3k}$. Let $N = 2^{3k}$, then we have 
    \begin{equation*}\label{eq:sandwich}
        \frac{1}{64}N < n \leq N.    
    \end{equation*}

    By \Cref{thm:alon}, there exists a triangle-free $(N,D,\lambda)$-graph $A$ with $D = 2^{k-1}(2^{k-1} - 1)$ and $\lambda \leq 10\cdot 2^{k} \leq 10 N^{1/3}.$ 
    We choose a set $V$ of $n$ vertices in $V(A)$ uniformly at random, and let $G:= A[V]$ be the subgraph induced by the set $V$. 

    Let $p=\frac{D}{N}$, then we have $\frac{1}{18} n^{-1/3}\leq p\leq \frac{1}{4} n^{-1/3}$.
    For each $v\in V$, the degree $d_{G}(v)$ follows hypergeometric distribution, \Cref{lem:chernoff} together with the union bound yields that with probability at least $1- n e^{-2\log n}>0$ all vertices of $G$ had degree $pn \pm 10 n^{1/3}\log^{1/2} n.$
    Fix such a random set $V$ and the graph $G=A[V]$.
    As we wish to obtain a regular subgraph, we plan to find a subgraph $F$ of $G$ which makes $G\setminus F$ regular.

    As $\lambda \leq 10 n^{1/3}$, the expander mixing lemma implies that $A$ is $(p,10 n^{1/3})$-bijumbled, its induced subgraph $G$ is also $(p,10 n^{1/3})$-bijumbled.
    We apply \Cref{lem:sponge-graph} to obtain two edge-disjoint subgraphs $R,S$ of $G$ satisfiying \ref{RS1}--\ref{RS3}. 
    Let $G_0 = G-(R\cup S)$, then \ref{RS1} implies that all vertices of $v$ had degree  $pn \pm 11 n^{1/3}\log^{1/2} n$ in $G_0$.
    Let 
    $$\delta = pn - 11 n^{1/3} \log^{1/2} n.$$
    In $G_0$, if two adjacent vertices $u, v$ have both degrees bigger than $\delta - d_S(u)$ and $\delta - d_S(v)$, respectively, we delete the edge $uv$. Repeat this until no such vertices exist, then we obtain a subgraph $G_1\subseteq G_0$ satisfying that 
    $\delta(G_1\cup S) \geq \delta$ and $W:= \{ v: d_{G_1\cup S}(v)> \delta\}$ is an independent set.
    For each $v\in W$, let $h(v) = d_{G_1\cup S}(v) - \delta$, then we have $h(v)\leq 22 n^{1/3} \log^{1/2}n$ by the definition of $\delta$.
    As $\Delta(G\setminus G_1)\leq 22 n^{1/3} \log^{1/2} n$, \Cref{lem:new-bijumbledness} ensures that $G_1$ is $(p,30n^{1/3}\log^{1/2} n)$-bijumbled, \Cref{lem: Hall} implies that there exists a subgraph $F$ of $G_1$ that satisfying the following.

    \begin{equation*}
        \text{$d_F(x)=h(x)$ for each $x\in W$, and  $d_F(y)\leq \log^{10}n$ for each $x\notin W$.}
    \end{equation*}
    The definition of $h$ implies that $G_2= G_1\setminus F$ satisfies $$\delta - \log^{10} n \leq \delta(G_2\cup S) \leq \delta.$$
    Again, $G_2$ is $(p,35n^{1/3}\log^{1/2} n)$-bijumbled as $\Delta(G\setminus G_2)\leq 35 n^{1/3} \log^{1/2} n$, \Cref{lem: spanning tree} ensures that $G_2$ contains a spanning tree $T$ with maximum degree at most $10$.
    By a result of~\cite{itai1977finding}, such a tree $T$ contains a parity subgraph $T'\subseteq T$, where $d_{T'}(v)$ and $d_{G_2\cup S}(v)$ has the same parity and $\Delta(T')\leq \Delta(T)\leq 10$.
    Then the graph $G^* = G_2\setminus T'$ disjoint from $R$ and $S$ satisfies $$\delta - \log^{10} n -10 \leq \delta(G^*\cup S) \leq \delta$$ and all vertices in $v$ has even degree in $G^*\cup S$.
    Let $d'$ be the largest even number less than $\delta - \log^{10} n -10$ and for each $v\in V(G)$, we let $f(v) =\frac{1}{2}( d_{G^*\cup S}(v) - d') \leq \log^{10} n + 12$.
    Then \ref{RS3} implies that there exists a subgraph $H\subseteq R\cup S$ disjoint from $G^*$ with $d_{H}(v) = d_S(v)- 2f(v)$.
    Then the final graph $G^*\cup H$ satisfies $$d_{G^*\cup H}(v) = d_{G^*\cup S}(v) - d_{S}(v)+ d_{H}(v) = d_{G^*\cup S}(v) - 2f(v) = d'$$
    for all $v\in V(G)$. Hence it is a $d'$-regular subgraph of $A[X]$. 
    As $d' \geq  pn - 20 n^{1/3} \log^{1/2}n$, we have $$\Delta(A[X]\setminus (G^*\cup H)) \leq 40 n^{1/3}\log^{1/2}n \text{ and } pn \geq \frac{1}{18} n^{2/3}.$$
    Thus, \Cref{lem: obtain lambda} implies that $G^*\cup H$ is an $(n, d', \lambda')$-graph with $d' \geq  pn - 20 n^{1/3} \log^{1/2}n \geq \frac{1}{20}n^{2/3}$ and $\lambda'\leq 41 n^{1/3} \log^{1/2} n$. This proves the theorem.
\end{proof}


\section*{Acknowledgement}
JK and HL are supported by the National Research Foundation of Korea (NRF) grant funded by the Korean government(MSIT) No. RS-2023-00210430. 
HL is supported by the Institute for Basic Science (IBS-R029-C4).


\printbibliography

\end{document}